\documentclass[reqno]{amsart}

\usepackage{amssymb, amsfonts, amsmath, amsthm, textcomp}

\newtheorem{theorem}{Theorem}[section]

\newtheorem{corollary}[theorem]{Corollary}

\newtheorem{conjecture}[theorem]{Conjecture}

\title[Hyponormality for the Ces\`{a}ro matrix of positive integer order]{A conjecture on hyponormality for the Ces\`{a}ro matrix of positive integer order}

\author{H. C. Rhaly Jr.}
\address{ H. C. Rhaly Jr.
\newline\hphantom{iii} 1081 Buckley Drive
\newline\hphantom{iii} Jackson, MS  39206, U.S.A.}
\email{rhaly@member.ams.org}


\begin{document}

\maketitle

\begin{abstract}
It is already known that the Ces\`{a}ro matrices of orders one and two are coposinormal, hyponormal operators on $\ell^2$.  Here it is shown that the Ces\`{a}ro matrices of order three and four are also coposinormal, hyponormal; the proofs employ posinormality, achieved by means of a diagonal interrupter, and elementary computational techniques from calculus.   A conjecture is then propounded for the Ces\`{a}ro matrix of positive integer order greater than four.
\end{abstract}

\section{Introduction}

Lat $a:\equiv \{a_n\}$ denote a sequence of nonnegative numbers with $a_0 > 0$, and take $S_n :\equiv \sum_{j=0}^n a_j > 0$.  The Norlund matrix $M_a :\equiv [m_{ij}]_{i,j \geq 0}$ is defined by
\[m_{ij} = \left \{ \begin{array}{lll}
a_{i-j}/S_i & for &  0 \leq j \leq i \\
0 & for & j > i .\end{array}\right.\]

\noindent The choice $a_n = \binom{n+ \alpha -1}{ \alpha - 1}$ generates $C(\alpha)$, the Ces\`{a}ro matrix of order $\alpha$ [\textbf{8}, p. 442]; here we will be concerned with the case in which $\alpha$ is a positive integer.

If $\mathcal{B}(H)$ denotes the set of all bounded linear operators on a Hilbert space $H$, then the operator $A \in \mathcal{B}(H)$ is \textit{hyponormal}  if \[< (A^*A - AA^*) f , f > \hspace{2mm}  \geq \hspace{2mm} 0\]  for all $f \in H$.  The operator $A \in \mathcal{B}(H)$ is said to be  \textit{posinormal} (see [\textbf{2]} , [\textbf{4}]) if \[AA^* = A^*PA\] for some positive operator $P \in \mathcal{B}(H)$, called the \textit{interrupter}.   The operator $A$ is \textit{coposinormal} if $A^*$ is posinormal.

First, consider $C(1)$, the Ces\`{a}ro matrix of order $1$  (take $\alpha = 1$), whose entries $m_{ij}$ are given by
\[m_{ij} = \left \{ \begin{array}{lll} \frac{1}{i+1} & for &  0 \leq j \leq i \\
0 & for & j > i .\end{array}\right.\]
\noindent In [\textbf{4}] it was observed that $C(1) \in B(\ell^2)$ satisfies \[C(1)C(1)^* = C(1)^*PC(1)\]  where \begin{equation*} P :\equiv diag \left \{ \frac{n+1}{n+2} : n \geq 0 \right \}; \end{equation*}  therefore, 
\begin{align*}  < (C(1)^*C(1) - C(1)C(1)^*) f , f > = < (C(1)^*C(1) - C(1)^*PC(1)) f , f > \end{align*}   \begin{align*}  = < (I-P)C(1) f , C(1)f >  \hspace{2mm}  \geq \hspace{2mm} 0 \end{align*}  
for all $f \in \ell^2$, so $C(1)$ is a hyponormal operator on $\ell^2$.  In this manner posinormality was used to give a proof of hyponormality for $C(1)$ that is different from an earlier one found in [\textbf{1}].

Next, consider $C(2)$, the Ces\`{a}ro matrix of order $2$, whose entries $m_{ij}$ given by
\[m_{ij} = \left \{ \begin{array}{lll} \frac{2(i+1-j)}{(i+1)(i+2)} & for &  0 \leq j \leq i \\
0 & for & j > i .\end{array}\right.\]
It was recently discovered in [\textbf{6}]  that $C(2) \in B(\ell^2)$  satisfies \[C(2)C(2)^* = C(2)^*PC(2)\]  where \begin{equation*} P :\equiv diag \left \{ \frac{(n+1)(n+2)}{(n+3)(n+4)} : n \geq 0  \right \}  \end{equation*} with \[I-P \geq 0,\] so $C(2)$ is also hyponormal on $\ell^2$.  (See also [\textbf{7}].)  The computations in [\textbf{6}] centered on coposinormality, and the diagonal form of $P$ emerged somewhat serendipitously from those computations.

The aim of this note is to apply a suitably altered version of this approach to the Ces\`{a}ro matrices of orders $3$ and $4$.  Rather than centering on coposinormality, here we aim directly for posinormality, armed now with a reasonable diagonal candidate to function as the interrupter $P$.  Besides posinormality, the proofs in the next two sections rely primarily on elementary techniques; namely, the use of 
\begin{itemize}

\item  formulas for sums of powers of positive integers and

\item telescoping sums to evaluate infinite series.

\end{itemize}

\noindent  After these proofs are complete, a conjecture will be propounded regarding posinormality (achieved with a specified diagonal interrupter), hyponormality, and coposinormality on $\ell^2$ for the more general Ces\`{a}ro matrix of positive integer order.

\section{The Ces\`{a}ro matrix of order $3$}

Under consideration in this section will be the Ces\`{a}ro matrix of order $3$, $M : \equiv C(3) \in B(\ell^2)$; the entries $m_{ij}$ of $M$ are given by 
\[m_{ij} = \left \{ \begin{array}{lll}
\frac{3(i+1-j)(i+2-j)}{(i+1)(i+2)(i+3)} & for &  0 \leq j \leq i \\
0 & for & j > i .\end{array}\right.\]   Before continuing on, we note that the range of M contains all the $e_n$'s from the standard orthonormal basis for $\ell^2$ since  
\[M \left[ \frac{(n+1)(n+2)(n+3)}{3!} (e_n-3e_{n+1}+3e_{n+2} - e_{n+3}) \right]  = e_n . \] 
   
In view of the considerations mentioned in the introduction, we first take  \begin{equation*} P :\equiv diag \left \{ \frac{(n + 1)(n + 2)(n + 3)}{(n + 4)(n + 5)(n + 6)} : n \geq 0 \right \}, \end{equation*} and then compute  $M^*PM$.  The ensuing calculations have been assisted by [\textbf{9}].  For $j \geq i$, the ($i$, $j$)-entry of $M^*PM$ is \begin{equation*} \sum_{k=0}^{\infty}\frac{9(j + 1 - i + k)(j + 2 - i + k)(k + 1)(k+2)}{(j + 1 + k)(j + 2  + k)(j + 3  + k)(j + 4 +  k)(j + 5  + k)(j + 6  + k)}. \end{equation*}

\noindent The series is telescoping, as can be seen by rewriting the summand as \begin{equation*} s(k) - s(k+1) \end{equation*} where  \begin{equation*} s(k) :\equiv   \frac{9(c_4k^4+c_3k^3+c_2k^2+c_1k+c_0)}{(j + 1 + k)(j + 2 + k)(j + 3 + k)(j + 4 + k)(j + 5 + k)} \end{equation*}
with \[ c_4 =1 ,  \hspace{2mm} c_3 = 8 - i + 3 j, \hspace{2mm} c_2 = \frac{1}{3} \cdot (71 - 15 i + i^2 + 57 j - 5 ij + 10 j^2) , \]

\[ c_1 = \frac{1}{6} \cdot (180 - 48 i + 6 i^2 + 236 j - 36 i j + i^2 j + 90 j^2 - 5 i j^2 + 
   10 j^3) , \]
 \noindent and 
   \[ c_0 = \frac{1}{30} \cdot (j + 4) (j + 5) (20 - 6 i + i^2 + 30 j - 5 i j + 10 j^2). \]

\noindent Consequently, for $j \geq i$, the ($i$, $j$)-entry of $M^*PM$ in simplified form is \begin{equation*} s(0) = \frac{9c_0}{(j + 1)(j + 2)(j + 3)(j + 4)(j + 5)} = \frac{3(20 - 6 i + i^2 + 30 j - 5 i j + 10 j^2)}{10(j + 1)(j + 2)(j + 3)}.\end{equation*}  

\noindent For $j \geq i$, the ($i$,$j$)-entry of $MM^*$ is  \begin{align}  \sum_{k=0}^{i} \frac{3(i-k+1)(i-k+2)}{(i+1)(i+2)(i+3)} \cdot \frac{3(j-k+1)(j-k+2)}{(j+1)(j+2)(j+3)} =  \end{align}  \begin{align*}  \frac{9}{(i+1)(i+2)(i+3)(j+1)(j+2)(j+3)} \cdot \sum_{k=0}^i k^4 -  d_3 k^3 + d_2 k^2 - d_1 k + d_0 .\end{align*} 
where \begin{align*} d_3 = 2(i+j+3), \end{align*} \begin{align*} d_2 = j^2 + 4ij + 9j + i^2 +9i + 13, \end{align*}  \begin{align*} d_1 = 2ij^2+3j^2+2i^2j+12ij+13j+3i^2+13i+12,  \end{align*} and  \begin{align*} d_0 = i^2j^2 + 3ij^2+2j^2+3i^2j+9ij+6j+2i^2+6i+4. \end{align*}  

\noindent Using the formulas for sums of powers of integers, expanding, and then factoring the result, one finds that the final summation in ($2.1$) becomes    \begin{align*}  \frac{1}{30} (i+1)(i+2)(i+3)(20 - 6 i + i^2 + 30 j - 5 i j + 10 j^2).\end{align*}  \noindent Substituting this result for that summation into ($2.1$) and simplifying, one obtains \begin{equation*} \frac{3(20 - 6 i + i^2 + 30 j - 5 i j + 10 j^2)}{10(j+1)(j+2)(j+3)}.  \end{equation*}  Thus it is seen that for $j \geq i$, the ($i$,$j$)-entry of $MM^*$  is the same as the ($i$, $j$)-entry of $M^*PM$;  by symmetry, the computations for $i \geq j$ are similar to those just presented, so it follows that $MM^*=M^*PM$.  Since it is clear that $I-P \geq 0$, the proof of hyponormality for $C(3)$ is complete, and the result is recorded below.
 
 \begin{theorem}  The Ces\`{a}ro matrix of order $3$ is a hyponormal operator on $\ell^2$.
 \end{theorem}
 
 The availability of the diagonal interrupter $P$ from the proof above makes the following corollary possible.

\begin{corollary}   If $M$ is the Ces\`{a}ro matrix of order $3$, then $M$ is coposinormal (i.e., $M^*$ is posinormal).
\end{corollary}

\begin{proof}  Apply [\textbf{5}, Theorem $1$(d)], using the fact that the interrupter $P$ in the proof above is invertible.
\end{proof}

\begin{corollary}  If $M$ is the Ces\`{a}ro matrix of order $3$, then both $M$ and $M^*$ are injective and have dense range with \[Ran (M) = Ran (M^*).\]
\end{corollary}

\begin{proof}  Since $M$ is posinormal, it follows from [\textbf{4}, Theorem $2.1$ and Corollary $2.3$] that \[Ran (M) \subseteq Ran (M^*)\]  and \[Ker (M) \subseteq Ker (M^*);\] since $M^*$ is also known to be posinormal (by the corollary above), the reverse inclusions must also hold; therefore,  \[Ker (M) = Ker (M^*)\] and \[Ran (M) = Ran (M^*).\]  It is easy to see that $Ker (M) = \{0\}$.  Consequently, both $M$ and $M^*$ are one-to-one, and both have dense range. 
\end{proof} 

\begin{corollary}  If $M$ is the Ces\`{a}ro matrix of order $3$, then $M^k$ is both posinormal and coposinormal for each positive integer $k$.
\end{corollary}

\begin{proof}  This follows from [\textbf{3}, Corollary $1$(b)].
\end{proof}

\section{The Ces\`{a}ro matrix of order $4$}

Under consideration here will be the Ces\`{a}ro matrix of order $4$, $M : \equiv C(4) \in B(\ell^2)$; the entries $m_{ij}$ of $M$ are given by 
\[m_{ij} = \left \{ \begin{array}{lll}
\frac{4(i+1-j)(i+2-j)(i+3-j)}{(i+1)(i+2)(i+3)(i+4)} & for &  0 \leq j \leq i \\
0 & for & j > i .\end{array}\right.\]   Note that the range of M contains all the $e_n$'s from the standard orthonormal basis for $\ell^2$ since \[ M \left[ \frac{(n+1)(n+2)(n+3)(n+4)}{4!} (e_n-4e_{n+1}+6e_{n+2} - 4e_{n+3}+e_{n+4}) \right] = e_n.\]  
Again, the ensuing calculations have been assisted by [\textbf{9}].
   
First take  \begin{equation*} P :\equiv diag \left \{ \frac{(n + 1)(n + 2)(n + 3)(n+4)}{(n + 5)(n + 6)(n + 7)(n+8)} : n \geq 0 \right \}, \end{equation*} and then compute  $M^*PM$.  For $j \geq i$, the ($i$, $j$)-entry of $M^*PM$ is \begin{equation*} \sum_{k=0}^{\infty}\frac{16(j + 1 - i + k)(j + 2 - i + k)(j + 3 - i + k)(k + 1)(k+2)(k+3)}{\prod_{t=1}^8 (j + t + k)}. \end{equation*}

\noindent The series is telescoping, as can be seen by rewriting the summand as \begin{equation*} s(k) - s(k+1) \end{equation*} where  \begin{equation*} s(k) :\equiv   \frac{16(c_6k^6+c_5k^5+c_4k^4+c_3k^3+c_2k^2+c_1k+c_0)}{(j + 1 + k)(j + 2 + k)(j + 3 + k)(j + 4 + k)(j + 5 + k)(j+6+k)(j+7+k)} \end{equation*}
with    
   \[ c_6=1 ,  \hspace{2mm} c_5 = \frac{1}{2} \cdot (33 - 3 i + 9 j), \hspace{2mm} c_4 = \frac{1}{2} \cdot (227 - 35 i + 2 i^2 + 130 j - 9 i j + 17 j^2) , \]
\[ c_3 = \frac{1}{4} \cdot (1644 - 321 i + 34 i^2 - i^3 + 1495 j - 178 ij + 7 i^2j + 
   414 j^2 - 21 ij^2 + 35 j^3) , \]
   
   \[ c_2 = \frac{1}{20}  \cdot (16250 - 3580 i + 520 i^2 - 30 i^3 + 21080 j - 3243 i j + 
   240 i^2 j - 3 i^3 j + 9325 j^2 \] \[ - 840 i j^2 + 21 i^2 j^2 + 
   1680 j^3 - 63 i j^3 + 105 j^4) , \]
   \[ c_1 = \frac{1}{20} \cdot (16290 - 3825 i + 670 i^2 - 55 i^3 + 28675 j - 5104 i j + 
   525 i^2 j - 14 i^3 j + 18170 j^2 \]  \[\hspace{7mm}  - 2186 i j^2 + 108 i^2 j^2 - 
   i^3 j^2 + 5250 j^3 - 364 i j^3 + 7 i^2 j^3 + 700 j^4 - 21 i j^4 + 
   35 j^5) , \]
   
  \noindent and
    \[ c_0 = \frac{1}{140} \cdot (j+5) (j+6) (j+7) (210 - 51 i + 10 i^2 - i^3 + 385 j - 
   70 i j + 7 i^2 j + 210 j^2 - 21 i j^2 + 35 j^3). \]

\noindent Consequently, for $j \geq i$, the ($i$, $j$)-entry of $M^*PM$ in simplified form is \begin{align*} s(0) = \frac{16c_0}{(j + 1)(j + 2)(j + 3)(j + 4)(j + 5)(j+6)(j+7)} \end{align*}  \begin{align*} = \frac{4(210 - 51 i + 10 i^2 - i^3 + 385 j - 
   70 i j + 7 i^2 j + 210 j^2 - 21 i j^2 + 35 j^3)}{35(j + 1)(j + 2)(j + 3)(j+4)}.\end{align*}  

\noindent For $j \geq i$, the ($i$,$j$)-entry of $MM^*$ is  \begin{align*}  \sum_{k=0}^{i} \frac{4(i-k+1)(i-k+2)(i-k+3)}{(i+1)(i+2)(i+3)(i+4)} \cdot \frac{4(j-k+1)(j-k+2)(j-k+3)}{(j+1)(j+2)(j+3)(j+4)} =  \end{align*}  \begin{align}  \frac{16}{\prod_{t=1}^4(i+t)(j+t)} \cdot \sum_{k=0}^i  k^6 - d_5 k^5 + d_4k^4 -  d_3 k^3 + d_2 k^2 - d_1 k + d_0 .\end{align}  

\noindent where \begin{align*} d_5 = 3(i+j+4), \end{align*} \begin{align*} d_4 =  3j^2 + 9ij + 30j +  3i^2 + 30i + 58, \end{align*}  \begin{align*} d_3 = j^3+9ij^2+24j^2+9i^2j+72ij+116j+i^3+24i^2+116i+144,  \end{align*}  \begin{align*} d_2 =  3ij^3+6j^3+9i^2j^2+54ij^2+69j^2+3i^3j+54i^2j+210ij+216j \end{align*}   \begin{align*} +6i^3+69i^2+216i+193 , \end{align*}     \begin{align*} d_1 = 3i^2j^3+12ij^3+11j^3+3i^3j^2+36i^2j^2+105ij^2+84j^2+12i^3j+105i^2j \end{align*}   \begin{align*} +264ij+193j+11i^3+84i^2+193i+132, \end{align*} and  \begin{align*} d_0 = i^3j^3+6i^2j^3+11ij^3+6j^3+6i^3j^2+36i^2j^2+66ij^2+36j^2+11i^3j+66i^2j \end{align*}  \begin{align*} +121ij+66j+6i^3+36i^2+66i + 36   . \end{align*}      
  
  \noindent Using the formulas for sums of powers of integers, expanding, and then factoring the result, one finds that the final summation in ($3.2$) becomes    \begin{align*}  \frac{ \prod_{t=1}^4(i+t)}{140} \cdot  (210 - 51 i + 10 i^2 - 
   i^3 + 385 j - 70 i j + 7 i^2 j + 210 j^2 - 21 i j^2 + 35 j^3)
.\end{align*}

  \noindent  Substituting this result for that summation into ($3.2$) and simplifying, one obtains \begin{equation*} \frac{4(210 - 51 i + 10 i^2 - 
   i^3 + 385 j - 70 i j + 7 i^2 j + 210 j^2 - 21 i j^2 + 35 j^3)}{35(j+1)(j+2)(j+3)(j+4)}.  \end{equation*}  Thus it is seen that for $j \geq i$, the ($i$,$j$)-entry of $MM^*$  is the same as the ($i$, $j$)-entry of $M^*PM$;  by symmetry, the computations for $i \geq j$ are similar to those just presented, so it follows that $MM^*=M^*PM$.  Since it is clear that $I-P \geq 0$, the proof of hyponormality for $C(4)$ is complete.  The result is recorded below.
 
 \begin{theorem}  The Ces\`{a}ro matrix of order $4$ is a hyponormal operator on $\ell^2$.
 \end{theorem}

Obvious analogues of Corollaries $2.2$\textemdash$2.4$ hold for the Ces\`{a}ro matrix of order $4$.

\section{Conjecture for the Ces\`{a}ro matrix of positive integer order}

In conclusion, a conjecture is offered for the general case.  Note that $C(N)$ and $P$ have already been shown to satisfy the conclusion below when $N=1$, $2$, $3$, and $4$. 
\begin{conjecture}  If $N > 4$ is a positive integer and $C(N)$ is the Ces\`{a}ro matrix of order $N$, with entries $m_{ij}$ given by \[m_{ij} = \left \{ \begin{array}{lll}
\frac{N \prod_{t=1}^{N-1} (i+t-j)}{\prod_{t=1}^N (i+t)} & for &  0 \leq j \leq i \\
0 & for & j > i , \end{array}\right.\]  
then $C(N)$ is a bounded, posinormal operator on $\ell^2$ with interrupter \[P=P(N) :\equiv diag \left \{\frac{ \prod_{t=1}^{N}(n+t)}{\prod_{t=N+1}^{2N}(n+t) } : n \geq 0 \right \},\] and, consequently, $C(N)$ is also hyponormal and coposinormal.

\end{conjecture}

There are obvious corollaries to this conjecture, similar to Corollaries $2.3$ and $2.4$.

\section{Update}

The following information has been provided by Billy E. Rhoades.

\vspace{2mm}
 
\noindent Sharma proved that every Hausdorff operator which is a bounded linear operator on $\ell^2$ is subnormal  [\textbf{10}, Theorem $2$].  Since every subnormal operator is hyponormal, and every hyponormal operator is posinormal, the Ces\`{a}ro matrix of order $\alpha$ is known to be posinormal and hyponormal on $\ell^2$  for all $\alpha \geq 1$.

\vspace{2mm}

Author's addendum: 

\vspace{2mm}

\noindent The work with the diagonal interrupter presented here can still be used to justify coposinormality (and corollaries), and that does not follow from [\textbf{10}].


\begin{thebibliography}{99}

\bibitem{1} A. Brown, P. R. Halmos, and A. L. Shields, {\it Ces\`{a}ro Operators}, Acta Sci. Math. (Szeged) { \bf 26} (1965), 125-137.

\bibitem{2} C. S. Kubrusly and B. P. Duggal, {\it On posinormal operators}, Adv. Math. Sci. Appl. {\bf 17} (2007), no. 1, 131-147.

\bibitem{3} C. S. Kubrusly, P. C. M. Vieira, and J. Zanni, {\it Powers of posinormal operators}, Oper. Matrices, {\bf 10} (2016), no. 1, 15--27.

\bibitem{4}  H. C. Rhaly Jr., {\it Posinormal operators}, J. Math. Soc. Japan \textbf{46} (1994), no. 4, 587 - 605.

\bibitem{5}  H. C. Rhaly Jr., {\it A superclass of the posinormal operators},  New York J. Math.,  {\bf 20} (2014), 497-506.  This paper is available via http://nyjm.albany.edu/j/2014/20-28.html.

\bibitem{6} H. C. Rhaly Jr.,  {\it The N\"{o}rlund operator on $\ell^2$ generated by the sequence of positive integers is hyponormal}, Bull. Belg. Math. Soc. Simon Stevin {\bf 22} (2015), no. 5, 737-742.

\bibitem{7}  H. C. Rhaly Jr.   {\it Supraposinormality and hyponormality for the generalized Ces\`{a}ro matrices of order two},  arXiv:1602.01408, February 3, 2016, preprint.


\bibitem{8}  B. E. Rhoades, {\it Using inclusion theorems to establish the summability of orthogonal series}, Approximation theory and spline functions (St. John's, Nfld., 1983), 441--453, NATO Adv. Sci. Inst. Ser. C Math. Phys. Sci., 136, Reidel, Dordrecht, 1984.

\bibitem{9}  W. A. Stein et al., Sage Mathematics Software (Version 6.10), The Sage Developers, 2015, http://www.sagemath.org.

\bibitem{10} N. K Sharma, {\it Hausdorff Operators}, Acta Sci Math. {\bf 35} (1973), 165-167.



\end{thebibliography}
\end{document}